\pdfoutput=1                      % per rendere testo accessibile a studenti ipovedenti

\documentclass[10 pt,a4paper]{article}

\usepackage[accsupp]{axessibility}    % per rendere testo accessibile a studenti ipovedenti

\usepackage{latexsym,amsfonts,amsmath,amsthm,amssymb, graphicx, mathrsfs, fancyhdr, makeidx, amscd,  color, mathdesign,fontenc,verbatim,stix}

\usepackage[all]{xy}
\usepackage{braket}

\usepackage{xcolor}

\newtheorem{theorem}{Theorem}
\newtheorem{acknowledgement*}{Acknowledgement}

\newtheorem{conjecture}[theorem]{Conjecture}
\newtheorem{corollary}[theorem]{Corollary}

\newtheorem{lemma}[theorem]{Lemma}

\newtheorem{remark}[theorem]{Remark}

\newcommand{\R}{\mathbb{R}}

\newcommand{\tr}{\operatorname{tr}}

\newcommand{\disc}{\mathrm{disc}}
\newcommand{\ess}{\mathrm{ess}}
\newcommand{\ra}{\rightarrow}
\newcommand{\haus}{\mathscr{H}}
\newcommand{\eps}{\varepsilon}
\newcommand{\PP}{\mathcal{P}}
\newcommand{\FF}{\mathscr{F}}

\newcommand{\disp}{\displaystyle}

\newcommand{\II}{\mathrm{II}}
\newcommand{\di}{\mathrm{d}}
\newcommand{\metric}{\langle \, , \, \rangle}

\newcommand{\USC}{\mathrm{USC}}

\newcommand{\lip}{\mathrm{Lip}}

\begin{document}

\author{Greg\'orio P. Bessa \and Luqu\'esio P.M. Jorge \and Luciano Mari}
\title{\textbf{On the principal eigenvalue of the truncated Laplacian, and submanifolds with bounded mean curvature}}
\date{\today}
\maketitle

\scriptsize \begin{center} Departamento de Matem\'atica, Universidade Federal do Cear\'a,\\
Campus do Pici, 60455-760 Fortaleza (Brazil)\\
E-mail: bessa@mat.ufc.br, ljorge@mat.ufc.br
\end{center}

\scriptsize \begin{center} Dipartimento di Matematica, Universit\`a degli Studi di Torino,\\
Via Carlo Alberto 10, I-10123 Torino (Italy)\\
E-mail: luciano.mari@unito.it
\end{center}

\normalsize

\begin{abstract}
In this paper, we study the principal eigenvalue $\mu(\FF_k^-,E)$ of the fully nonlinear operator
	\[
	\FF_k^-[u] = \PP_k^-(\nabla^2 u) - h |\nabla u| 
	\]
on a set $E \Subset \R^n$, where $h \in [0,\infty)$ and $\PP_k^-(\nabla^2 u)$ is the sum of the smallest $k$ eigenvalues of the Hessian $\nabla^2 u$. We prove a lower estimate for $\mu(\FF_k^-,E)$ in terms of a  generalized Hausdorff measure $\haus_\Psi(E)$, for suitable $\Psi$ depending on $k$, moving some steps towards the conjecturally sharp estimate 
	\[
	\mu(\FF_k^-,E) \ge C \haus^k(E)^{-2/k}. 
	\]
The theorem is used to study the spectrum of bounded submanifolds in $\R^n$, improving on our previous work in the direction of a question posed by S.T. Yau. In particular, the result applies to solutions of Plateau's problem for CMC surfaces. 
\end{abstract}

\begin{center}
To Renato Tribuzy, on the occasion of his 75th birthday, with great admiration.
\end{center}

\tableofcontents

\section{Introduction}

It is a great pleasure for us to dedicate our work to Renato Tribuzy on the occasion of his 75th birthday, in recognition for his outstanding work to shape the field of Differential Geometry in Brazil, especially in the Amazon region.\par 
This note is about the spectral properties of some fully nonlinear, degenerate operators of geometric interest in $\R^n$. For $w \in C^2(\R^n)$, we let 
	\[
	\lambda_1(\nabla^2 w) \le \lambda_2(\nabla^2 w) \le \ldots \le \lambda_n(\nabla^2 w)
	\]
be the eigenvalues of the Hessian $\nabla^2 w$, in increasing order, and given $k \in \{1, \ldots, n\}$ we define
$$
\begin{array}{lcl}
\PP^-_k(\nabla^2 w) & \doteq & \lambda_1(\nabla^2 w) + \ldots + \lambda_k(\nabla^2 w), \\[0.2cm]
\PP^+_k(\nabla^2 w) & \doteq & \disp \lambda_{n-k+1}(\nabla^2 w) + \ldots + \lambda_n(\nabla^2 w). 
\end{array}
$$
We do not consider the case $k=n$, for which $\PP_n^+=\PP_n^- = \Delta$, and hereafter restrict to $k \in \{1,\ldots, n-1\}$ unless otherwise specified. Given $h \in \R^+_0 \doteq [0, \infty)$, we then consider the operators 
	\begin{equation}\label{def_Fkpm}
	\FF_k^+[w] \doteq \PP_k^+(\nabla^2 w) + h |\nabla w|, \qquad \FF_k^-[w] \doteq \PP_k^-(\nabla^2 w) - h|\nabla w|.
	\end{equation}
Both $\PP_k^\pm$ and $\FF_k^\pm$ naturally arise in Differential Geometry, especially in the theory of submanifolds. For instance, they appeared in the level set formulation of the mean curvature flow with higher codimension \cite{ambrosiosoner}, and to formulate partially positive Ricci curvature conditions suited to obtain Morse-theoretic results \cite{sha,wu}; they have been used in connection with barrier principles for submanifolds with higher codimension in \cite{joto} and later in \cite{white,white_2,galimame}; in the (almost) complex or calibrated setting, they are ubiquitous in the study of plurisubharmonic functions and in  potential theory \cite{HL_levi,HL_almostcom}. However, despite the many applications, only in recent years the analytic properties of $\PP_k^\pm$ and $\FF_k^\pm$ have systematically been investigated. In this respect, we quote \cite{HL_dir,HL_plurisub} by R. Harvey and B. Lawson, \cite{obsi} by A.M. Oberman and L. Silvestre, \cite{caffalini} by L. Caffarelli, Y.Y. Li and L. Nirenberg, and \cite{bigais_1,bigais_existence} by I. Birindelli, G. Galise and H. Ishii. Following \cite{bigais_1}, we name $\PP_k^\pm$ \emph{truncated Laplacians}.
% 
%$$
%\begin{array}{lcl}
%\PP^-_k(\nabla^2 w) & = & \disp \frac{\lambda_1(\nabla^2 w) + \ldots + \lambda_k(\nabla^2 w)}{k} \\[0.2cm]
%\PP^+_k(\nabla^2 w) & = & \disp \frac{\lambda_{n-k+1}(\nabla^2 w) + \ldots + \lambda_n(\nabla^2 w)}{k} 
%\end{array}
%$$
%\begin{remark}
%\emph{Note that, in Harvey-Lawson's notation, for $c \in \R$ the dual of the subset 
%$$
%F_k = \Big\{ \PP_k^- w \ge - c w \Big\} \qquad \text{is} \qquad \widetilde{F}_k = \Big\{ \PP_k^+ w \ge -cw\Big\}.
%$$
%}
%\end{remark}

Denote with $\USC(A)$ the set of upper-semicontinuous functions on a set $A \subset \R^n$. Following \cite{beniva}, there exist, at least, two slightly different notions of principal eigenvalue of $\FF_k^\pm$ on an open set $\Omega \subset \R^n$:
$$
\begin{array}{l}
\disp \mu(\FF_k^\pm, \Omega) \doteq \sup \Big\{ c \in \R : \exists w \in \USC(\Omega), \ w < 0 \ \ \text{on $\Omega$, } \, \FF^\pm_k[w] + c w \ge 0 \ \text{on } \, \Omega \Big\}, \\[0.4cm]
\disp \bar{\mu}(\FF_k^\pm, \Omega) \doteq \sup \Big\{ c \in \R :  \exists w \in \USC(\overline{\Omega}), \ w < 0 \ \ \text{on $\overline{\Omega}$, } \, \FF^\pm_k[w] + c w \ge 0 \ \text{on } \, \Omega \Big\}.
\end{array}
$$
Hereafter in this paper, inequality $\FF^\pm_k [w] + c w \ge 0$ is meant to hold in the viscosity sense. Note that $0 \le \bar{\mu}(\FF_k^\pm,\Omega) \le \mu(\FF_k^\pm,\Omega)$, since negative constants are admissible as $w$.

\begin{remark}\label{rem_posinega}
\emph{Customarily, principal eigenvalues are also defined in terms of positive supersolutions of $\FF^\pm_k [w] + c w = 0$. However, in view of the identity $\FF_k^-[-w] = -\FF_k^+[w]$, this doesn't introduce further constants of interest, since for instance $\mu(\FF_k^\pm,\Omega)$ can equivalently be defined as 
\[
\sup \Big\{ c \in \R : \exists w \in \mathrm{LSC}(\Omega), \ w > 0 \ \ \text{on $\Omega$, } \, \FF^\mp_k[w] + c w \le 0 \ \text{on } \, \Omega \Big\}.
\]
}
\end{remark}

%Up to replacing $w$ with $-w$, we can equivalently define $\mu_k^-, \mu_k^+$ as follows:
%$$
%\begin{array}{l}
%\disp \mu_k^-(\Omega) = \sup \Big\{ c \in \R \ \ : \ \ \exists w > 0 \ \ \text{on $\Omega$ solving} \ \ \FF^+_k[w] + c w \le 0 \ \text{on } \, \Omega \Big\}, \\[0.4cm]
%\disp \mu_k^+(\Omega) = \sup \Big\{ c \in \R \ \ : \ \ \exists w > 0 \ \ \text{on $\Omega$ solving} \ \ \FF^-_k[w] + c w \le 0 \ \text{on } \, \Omega \Big\}.
%\end{array}
%$$
For $E \subset \R^n$, define
$$
\mu(\FF_k^\pm,E) \doteq \sup \Big\{ \mu(\FF_k^\pm,\Omega) \ : \ \text{$\Omega \subset \R^n$ open, $E\subset \Omega$}\Big\}.
$$
and $\bar \mu(\FF_k^\pm,E)$ accordingly. The purpose of the present paper is to investigate possible lower bounds for $\bar \mu(\FF_k^\pm,E)$ depending on the size of $E$, in the spirit of the Faber-Krahn inequality
	\[
	\mu(\Delta, \Omega) \ge \left[\frac{\mu(\Delta,B)}{|B|^{-2/n}}\right] |\Omega|^{-\frac{2}{n}}
	\]
where $\Omega \Subset \R^n$ has smooth boundary, and $B$ is a ball with $|B|= |\Omega|$. For second order, uniformly elliptic operators in trace form
	\[
	Lw = a_{ij} \partial^2_{ij} w + b_i \partial_i w
	\]
with bounded, measurable coefficients $a_{ij}=a_{ji}$ and $b_i$ on $\R^n$, works of H. Berestycki, L. Nirenberg and S. Varadhan in \cite[Thm. 2.5]{beniva} and X. Cabr\'e in \cite{cabre} established the estimate 
	\begin{equation}\label{eq_lowerL}
	\mu(L,\Omega) \ge C |\Omega|^{-\frac{2}{n}} \qquad \forall \, \Omega \Subset \R^n \ \text{ open,} 
	\end{equation}
for some constant $C>0$ only depending on the ellipticity constants of $a_{ij}$, on $\|b_i\|_{L^n(\Omega)}$ and on an upper bound for $|\Omega|^{1/n}$. The case of $L$ in divergence form (with bounded, measurable coefficients) was shown before by H. Brezis and P.-L. Lions in \cite{brezislions}. For fully nonlinear operators which are $1$-homogeneous and uniformly elliptic, we are not aware of estimates like  \eqref{eq_lowerL}. However, a weaker result with a lower bound depending on $|\Omega|^{-1/n}$ can be found in \cite[Prop. 4.8]{qs}.
%	$F(x,u, \nabla u, \nabla^2 u)$ which are $1$-homogeneous and uniformly elliptic, in the sense that 
%	\[
%	\begin{array}{l}
%	F(x,tu,t\nabla u,t \nabla^2 u) = t F(x,u,\nabla u,\nabla^2 u) \qquad \forall \, t \in \R^+_0 \\[0.2cm]
%	.....................
%	\end{array}
%	\]
%We are only aware of the following etimate by A. Quaas ans B. Sirakov \cite{qs}, where they prove
%	\[
%	\mu( F, \Omega) \ge C |\Omega|^{-\frac{1}{n}}
%	\]
%\tcr{(careful with scaling)}.	   

\vspace{0.2cm}

Inequalities of the type in \eqref{eq_lowerL} for $\FF_k^\pm$ seem quite difficult to achieve. Among the issues to overcome, we stress that the proofs of \eqref{eq_lowerL} are based on the ABP method and that, to our knowledge, sharp ABP inequalities tailored to the degenerate elliptic operators $\FF_k^\pm$ are yet to be formulated; their lack also helps to explain the absence of regularity results for $\PP_k^\pm$ when $k \not\in \{1,n\}$ (for $k=1$, see \cite{obsi,bigais_1}). A series of unusual phenomena for $\FF_k^\pm$ was first pointed out by I. Birindelli, G. Galise and H. Ishii in \cite{bigais_1}, and the results therein reveal the prominent role played by a boundedness condition for $h$ related to the diameter of $\Omega$. Let $k \in \{1, \ldots, n-1\}$, and let $R$ such that $\Omega \subset B_R$. In \cite[Cor. 4.2 and Prop. 4.3]{bigais_1}, the authors proved that 
	\[
	hR < k \qquad \Longrightarrow \qquad \left\{ \begin{array}{ll}
	\bar \mu( \FF_k^-, \Omega) \ge \frac{2(k-hR)}{R^2} \\[0,4cm]
	\bar \mu(\FF_k^+,\Omega) = +\infty.
	\end{array}\right.
	\]
Hence, searching for lower bounds for the principal eigenvalue of $\FF_k^+$ is meaningless, at least if $hR < k$, and hereafter we will focus on $\FF_k^-$. Condition $hR < k$ is sharp to guarantee a positive lower bound for $\bar \mu( \FF_k^-, \Omega)$. Indeed, as proved in \cite[Ex, 4.9]{bigais_1}, for each $\eps > 0$ small enough the annulus 
	\[
	\Omega_\eps = B_{3\pi/2 + \eps}\setminus \overline{B}_{3\pi/2 - \eps} \subset \R^n 
	\]
satisfies $\bar \mu(\FF_k^-,\Omega_\eps) = 0$ with the choice $h = k/(3\pi/2)$. Note that conditions $\Omega_\eps \subset B_R$ and $hR \le k$ barely fail to be simultaneously satisfied. Also, the example shows that the $n$-dimensional measure of $E$ is not expected to control $\bar \mu(\FF_k^-,E)$. 

A surprising fact is the validity of \emph{reversed} Faber-Krahn inequalities for the operator $\PP_1^-$. As conjectured in \cite{bigais_FK} building on results for multidimensional rectangles, and proved in \cite{parosa}, $\mu(\PP_{1}^-,\Omega)$ is \emph{maximized} by the ball among domains with the same fixed diameter  (please mind the conventions in \cite{parosa} and recall Remark \ref{rem_posinega}). Consequently, it is also maximized by the ball among domains with the same fixed volume.\par
\vspace{0.2cm}

We are ready to state our main result. To this aim, we recall that given a continuous, non-decreasing $\Psi : [0,c) \to \R^+_0$ with $\Psi(0)=0$, the generalized Hausdorff measure of order $\Psi(t)$ is defined by
			\[
				\haus_\Psi(E) \doteq \lim_{\delta \to 0^+} \inf \left\{ \sum_j \Psi(r_j) \ : \ E \subset \bigcup_{j=1}^\infty B_{r_j}(x_j), \ r_j \le \delta \right\}.
			\]
If $\Psi(t) = t^k$ then $\haus_\Psi$ is, up to an inessential constant, the standard Hausdorff $k$-dimensional measure $\haus^k$. We underline the inequality
	\[
	\haus^2(E) \le C \haus_\Psi(E) \qquad \text{where } \, \Psi(t) = t^2| \log(R/t)|, \ R \in \R^+,
	\]	
for some constant $C= C(R,c)$. 			

\begin{theorem}\label{teo_maingoal_k2}
Let $E \subset \R^n$ be a compact set of diameter $\mathrm{diam}(E) < R$, Fix $k \in \{1,\ldots, n-1\}$ and let $h \in \R^+_0$ satisfying $hR < k$. Then, there exists a constant $C= C(n,k,hR)$ such that 
\begin{equation}\label{eq_thegoal}
\bar\mu(\FF_k^-,E) \ge \frac{C}{\haus_\Psi(E)}, \qquad \text{where } \, \Psi(t) = \left\{ \begin{array}{ll}
	Rt & \quad \text{if } \, k=1 \\[0.2cm]
	t^2 |\log(R/t)| & \quad \text{if } \, k=2 \\[0.2cm]
	t^2 & \quad \text{if } \, k \ge 3.
	\end{array}\right.
\end{equation}
In particular, $\bar \mu(\FF_k^-,E) = +\infty$ whenever $\haus_\Psi(E) = 0$. 
\end{theorem}

\begin{remark}
\emph{We stress that inequality \eqref{eq_thegoal} is scale-invariant for each $k$, due to the presence of $R$ in the definition of $\Psi(t)$. 
}
\end{remark}

\begin{remark}\label{rem_constantC}
\emph{The constant $C$ can be bounded from below uniformly in terms of $k,n$ and a of a lower bound for $k - hR$.
}
\end{remark}

It is reasonable to guess that the lower bound for $\bar \mu(\FF_k^-,E)$ in terms of the Hausdorff $k$-measure, that we proved for $k=1$, be obtainable also for $k > 1$. If so, also the case $k=2$ of our Theorem would  be nearly sharp, failing only by a logarithmic term. We propose the following   

\begin{conjecture}\label{conj_mainana}
Suppose that $E \subset \R^n$ is a compact subset of diameter $< R$, fix $k \in \{1,\ldots, n-1\}$ and $h \in \R^+$ satisfying $hR < k$.  Then, there exists a constant $C>0$ depending on $n,k, hR$ and on an upper bound for $\haus^k(E)$ such that
	\[
	\mu ( \FF_k^-, E) \ge C \haus^k(E)^{-\frac{2}{k}}.
	\]
In particular, if $\haus^k(E) = 0$ then $\mu(\FF_k^-,E) = +\infty$.
\end{conjecture}

It may be possible that condition $\mathrm{diam}(E) < R$ could be weakened to $E \subset B_R(o)$ for some $o \in \R^m$.

\subsection*{A geometric application} 

A source of motivation for the present paper comes from the theory of minimal submanifolds in $\R^n$. Indeed, the note arises from the desire to put the main result in \cite{beluma} into a more general perspective, explaining how it descends from an estimate for the principal eigenvalue of $\FF_k^-$. At the same time, we improve on \cite{beluma} on various aspects, in particular for submanifolds with nonzero mean curvature. \par
In \cite{beluma}, we addressed a question of S.T. Yau about the discreteness of the spectrum of the Laplacian  of some striking examples of bounded, complete minimal surfaces constructed after N. Nadirashvili's counterexample to an old conjecture of E. Calabi \cite{nadirashvili}. We recall that the spectrum $\sigma(-\Delta)$ of the Laplace-Beltrami operator on a manifold $M$ is said to be discrete if it only contains a divergent sequence of eigenvalues, each of them with finite multiplicity. For instance, this happens if $M$ is the interior of a compact submanifold with smooth boundary. In this case, clearly $M$ is not complete as a metric space. On the other hand, complete minimal surfaces which are well-behaved, in the sense that they have finite density at infinity: 
	\[
	\lim_{r \to \infty} \frac{|M\cap \mathbb{B}_r|}{r^2} < \infty, \qquad \mathbb{B}_r \subset \R^3 \ \text{a ball}, 
	\]
satisfy $\sigma(-\Delta) = \R^+_0$ by \cite[Thm. 1]{limamovi}. Therefore, complete manifolds with discrete spectrum are expected to exhibit a pathological behaviour, and the examples arisen after Nadirashvili's work are good candidates to have discrete spectrum. Nadirashvili constructed a complete minimal surface $M^2 \to \R^3$ which is bounded in $\R^3$, and his method, a far reaching extension of that of L. Jorge and F. Xavier in \cite{jx}, inspired an entire literature: in particular, highly nontrivial refinements  \cite{mm_1,mm_2,fmm,lmm_1,mn} and entirely new methods \cite{adfl,af,agl,al} enabled to construct complete, bounded minimal surfaces whose behaviour at infinity is controlled in some way. To be more precise, given an immersion $\varphi : M \ra \R^n$, we define the limit set 
$$
\lim \varphi = \Big\{ p \in \R^n \ : \ p = \lim_j \varphi(x_j) \ \ \text{for some divergent sequence $\{x_j\} \subset M$}\Big\}.
$$
Here, $\{x_j\}$ is said to be divergent if it eventually lies outside every fixed compact set of $M$. Note that, if $\varphi(M) \subset \Omega$ for some domain $\Omega$, then $\lim \varphi \subset \overline{\Omega}$. If $\lim \varphi \subset \partial \Omega$, we say that $M$ is proper in $\Omega$. After Nadirashvili's work,  proper examples in convex sets were constructed in \cite{mm_1,mm_2,fmm}, examples with a control on the conformal structure of $M$ in \cite{af,agl,al}, and examples with a control of the Hausdorff dimension of $\lim \varphi$, in the sense that $\dim_\haus(\lim \varphi) = 1$, in \cite{mn,adfl}. They motivated our criterion in \cite[Thm. 2.4]{beluma}, which we refine in the present note.\par
To state the result, we introduce some terminology. For a $(2,0)$-tensor $A$ with eigenvalues $\{\lambda_j(A)\}_{j=1}^n$ in increasing order, and given $k \in \{1,\ldots, n\}$, we write 
	\[
	\PP_k^-(A) \doteq \lambda_1(A) + \ldots + \lambda_k(A). 
	\]
\begin{remark}\label{rem_minmax}
\emph{Note that $\PP_k^-(A)$ can be characterized as follows:
	\[
	\PP_k^-(A) = \inf \Big\{ \tr_W A \colon \ W \, \text{a $k$-dimensional subspace} \Big\},
	\]
where, taken an orthonormal basis $\{e_i\}$ for $W$, $\tr_W A \doteq \sum_{i=1}^k A(e_i,e_i)$. 
}
\end{remark}	
	
Given a $k$-dimensional immersed submanifold $\varphi : M \to \R^n$, we denote with ${\bf H}$ the unnormalized mean curvature vector, that is, the trace of the second fundamental form of $M$. Let $\Omega \subset \R^n$ be an open subset, and let $\Lambda_{k-1},\Lambda_k \in \R$. We say that $\partial \Omega$ satisfies 
	\[
	\inf_{\partial \Omega} \PP_k^-(\II_{\partial \Omega}) \ge \Lambda_k, \qquad \inf_{\partial \Omega} \PP_{k-1}^-(\II_{\partial \Omega}) \ge \Lambda_{k-1}
	\]
in the barrier sense if, for each $x \in \partial \Omega$ and $\eps>0$, there exists a supporting smooth  hypersurface $S$ such that $S \cap \Omega = \emptyset$, $x \in S$ and the second fundamental form $\II_S$ of $S$ in the direction pointing towards $\Omega$ satisfies both of the inequalities
	\[
	\PP_k^-(\II_{S})(x) \ge \Lambda_k -\eps, \qquad \PP_{k-1}^-(\II_S)(x) \ge \Lambda_{k-1} -\eps.
	\]
For instance, by using hyperplanes as supporting hypersurfaces, a convex set $\Omega$ satisfies
	\[
	\inf_{\partial \Omega} \PP_k^-(\II_{\partial \Omega}) \ge 0 \qquad \inf_{\partial \Omega} \PP_{k-1}^-(\II_{\partial \Omega}) \ge 0,
	\]
and a domain that can be written as the intersection of balls of radius $R$ satisfies	
	\[
	\inf_{\partial \Omega} \PP_k^-(\II_{\partial \Omega}) \ge \frac{k}{R} \qquad \inf_{\partial \Omega} \PP_{k-1}^-(\II_{\partial \Omega}) \ge \frac{k-1}{R}.
	\]
Given $\overline{\Lambda}_k \in \R \cup \{-\infty\}$, we also say that 
	\[
	\inf_{\partial \Omega} \PP_k^-(\II_{\partial \Omega}) > \overline{\Lambda}_k
	\]
if there exists $\Lambda_k > \overline{\Lambda}_k$ such that $\inf_{\partial \Omega} \PP_k^-(\II_{\partial \Omega}) \ge \Lambda_k$. \par
	
We are ready to state	
		
\begin{theorem}\label{teo_main_geom}
Let $\varphi : M \rightarrow \R^n$ be a bounded immersed submanifold of dimension $k \ge 2$, contained in a relatively compact domain $\Omega$ with diameter $R$. Assume that the mean curvature vector $\mathbf{H}$ of $M$ satisfies $R\|\mathbf{H}\|_\infty < k$. Define
	\[
\Psi(t) = \left\{ \begin{array}{ll} t^2| \log(R/t)| & \quad \text{if } \, k = 2 \\[0.2cm]
t^2 & \quad \text{if } \, k \ge 3.
\end{array}\right. 
	\]
Assume that either 
	\begin{itemize}
	\item[(i)] $\haus_\Psi(\lim \varphi) = 0$, or 
	\item[(ii)] $\haus_\Psi(\lim \varphi \cap \Omega) = 0$ and the second fundamental form $\II_{\partial \Omega}$ of $\partial \Omega$ in the inward direction satisfies 
	\begin{equation}\label{eq_barrierbordo}
	\inf_{\partial \Omega} \PP_k^-(\II_{\partial \Omega}) > \|\mathbf{H}\|_{\infty}, \qquad \inf_{\partial \Omega} \PP_{k-1}^-(\II_{\partial \Omega}) > -\infty
	\end{equation}
in the barrier sense. 
	\end{itemize}
Then, the spectrum of the Laplace-Beltrami operator on $M$ is discrete. 
\end{theorem}

\begin{remark}\label{rem_tech}
\emph{Clearly, if $\partial \Omega$ is $C^2$, \eqref{eq_barrierbordo} is equivalent to $\PP_k^-(\II_{\partial \Omega}) > \|\mathbf{H}\|_{\infty}$ on $\partial \Omega$, that was the condition stated in \cite{beluma}. Besides the weaker regularity assumed on $\partial \Omega$, Theorem \ref{teo_main_geom} improves on \cite[Thm. 2.4]{beluma} when $\mathbf{H} \not \equiv 0$ for each $k$. First, condition $R\|{\bf H}\|_\infty < k$ is weaker than $R \|{\bf H}\|_\infty < k-1$, which was required in \cite{beluma}. Second, when 
	\[
	\theta \doteq k-1-R \|{\bf H}\|_\infty \in (0,1) 
	\]
(which is automatic if $k=2$ and $\mathbf{H} \not\equiv 0$), condition $\haus_\Psi(\lim \varphi \cap \Omega) = 0$ was replaced by the stronger 
	\[
	\haus^{\theta+1}(\lim \varphi \cap \Omega) = 0,
	\]
with the somehow puzzling feature that $R\|{\bf H}\|_\infty$ appeared to bound the exponent of the Haudorff dimension. The possibility to get better dimensional conditions for $\lim \varphi \cap \Omega$ depends on Lemmas \ref{lem_pk discend} and \ref{basicconstr_Pk} for the operator $\FF_k^-$, which may have an independent interest. 
}
\end{remark}

The above result is particularly effective when $k=2$, since for instance it can be applied to any of the examples in \cite{mm_1,mm_2,fmm,mn,adfl} to answer Yau's question, as done in \cite{beluma}. Also, Theorem \ref{teo_main_geom} applies to solutions of Plateau's problem for (parametrized) surfaces with constant mean curvature (see \cite{struwe} for a detailed account), and our condition on $h$ is almost sharp: indeed, interestingly, for a rectifiable Jordan curve $\gamma \subset B_R$ inequality $hR \le 2$ turns out to be sharp to guarantee the existence of a topological disk with constant mean curvature $h$ and boundary $\gamma$, in the sense that if $hR>2$ then there exists $\gamma \subset B_R$ such that Plateau's problem has no solution with mean curvature $h$ (cf. \cite{heinz}). The next result was shown in \cite[Cor. 2.6]{beluma} for minimal surfaces. 

\begin{corollary}
Let $\gamma : \mathbb{S}^1 \to \R^n$ be a Jordan curve with $\mathrm{diam}(\gamma(\mathbb{S}^1)) \le R$ and 
\[
\haus_{\Psi}\big( \gamma(\mathbb{S}^1)\big) = 0, \qquad \Psi(t) = t^2|\log (R/t)|. 
\]
Fix $h \in \R^+_0$ satisfying $hR < 2$. Then, every solution of Plateau's problem for surfaces with constant (unnormalized) mean curvature $h$ and boundary $\gamma$ has discrete spectrum. 
\end{corollary}

The geometric counterpart of Conjecture \ref{conj_mainana} is the following

\begin{conjecture}\label{conj_maingeom}
Let $\varphi : M \rightarrow \R^n$ be a bounded immersed submanifold of dimension $k \ge 2$, contained in a relatively compact domain $\Omega$ with diameter $R$. Assume that the mean curvature vector $\mathbf{H}$ of $M$ satisfies $R\|\mathbf{H}\|_\infty < k$, and that either
	\begin{itemize}
	\item[(i)] $\haus^k(\lim \varphi) = 0$, or 
	\item[(ii)] $\haus^k(\lim \varphi \cap \Omega) = 0$ and the second fundamental form $\II_{\partial \Omega}$ of $\partial \Omega$ in the inward direction satisfies 
	\[
	\inf_{\partial \Omega} \PP_k^-(\II_{\partial \Omega}) > \|\mathbf{H}\|_{\infty}, \qquad \inf_{\partial \Omega} \PP_{k-1}^-(\II_{\partial \Omega}) > -\infty
	\]
in the barrier sense. 
	\end{itemize}
Then, the spectrum of the Laplace-Beltrami operator on $M$ is discrete. 
\end{conjecture}

A word of warning: our proof of Theorem \ref{teo_main_geom} based on Theorem \ref{teo_maingoal_k2} could  easily be adapted to prove the geometric Conjecture \ref{conj_maingeom} from Conjecture \ref{conj_mainana} \emph{only in case} $(i)$. Case $(ii)$ seems to be subtler.

\section{Proof of Theorem \ref{teo_maingoal_k2}}

We start with the following ODE Lemma.

\begin{lemma}\label{lem_pk discend}
Let $k \in \{1,\ldots,n\}$, $R \in \R^+$ and $h,h^* \in \R^+_0$ satisfying 
\[
hR < k, \qquad h^* \ge \max \left\{ h, \frac{h}{k-hR}\right\}.
\]
Let $\xi \in C(\R^+)$ be non-increasing, non-negative and such that 
\begin{equation}\label{inte_w}
\int_{0^+} t^{k-1}\xi(t) \di t < \infty,
\end{equation}
and let $\psi \in C^2\big((0,R)\big)$ solve
\begin{equation}\label{eq_h}
\left\{\begin{array}{l}
\disp \big(t^{k-1}e^{-h^*t}\psi'\big)' = e^{-h^*t} t^{k-1} \xi \qquad \text{on } \, (0, R), \\[0.2cm]
\disp \lim_{t \ra 0}\big(t^{k-1}\psi'(t)\big) = 0.
\end{array}\right.  
\end{equation}
Fix $x_0 \in \R^n$ and set $r(x)=|x-x_0|$. Then, the function $w(x) = \psi(r(x))$ satisfies  
\begin{equation}\label{Pk_h}
%\begin{array}{ll}
\disp \FF_k^-[w] \doteq \PP_k^-(\nabla^2 w) - h|\nabla w| \ge \frac{\xi(r)}{1+h^*R}  \qquad \text{on } \, B_{R}(x_0) \backslash \{x_0\}.
%\disp \PP_k^- w \ge \xi(r) & \quad \text{on } \, B_{R}(x_0), \ \ \text{in viscosity sense.} 
%\end{array}
\end{equation}
Moreover, $w(x) \in C^2(B_R(x_0))$ and the inequality holds pointwise on the entire $B_R(x_0)$ provided that $\xi \in C(\R^+_0)$.
\end{lemma}

\begin{proof}
From 
$$
\nabla^2 w = \psi'' \di r \otimes \di r + \psi' \nabla^2 r = \left( \psi'' - \frac{\psi'}{r}\right) \di r \otimes \di r + \frac{\psi'}{r} \metric,
$$
the eigenvalues of $\nabla^2 w$ are $\psi''(r)$ with multiplicity 1, and $\psi'(r)/r$ with multiplicity $(n-1)$. Note that, expanding \eqref{eq_h}, 
\begin{equation}\label{prop_h_2}
\psi''(t) + \frac{k-1}{t}\psi'(t) = \xi(t) + h^* \psi'(t) \qquad \text{on } \, (0,R).
\end{equation}
Integrating \eqref{eq_h} on $(\eps,t)$, we get
\[
\psi'(t) = \frac{e^{h^*t}}{t^{k-1}} \left\{ e^{-h^*\eps}\eps^{k-1}\psi'(\eps) + \int_\eps^t e^{-h^*s}s^{k-1} \xi(s) \di s \right\}.
\]
Since the last term in brackets has a finite limit as $\eps \ra 0$ by \eqref{inte_w}, and because of the limit condition in \eqref{eq_h}, 
\begin{equation}\label{prop_h}
\psi'(t) = \disp \frac{e^{h^* t}}{t^{k-1}}\int_0^t e^{-h^*s}s^{k-1} \xi(s) \di s \ge 0 \qquad \text{on } \, (0,R).
\end{equation}
We claim that 
\begin{equation}\label{eq_claim}
\psi''(t) \le (1 + h^*R) \frac{\psi'(t)}{t} \qquad \text{on } \, (0,R).
\end{equation}
Indeed, since $\xi(s)e^{-h^*s}$ is non-increasing, 
\[
\psi'(t) \ge \frac{e^{h^*t}}{t^{k-1}} \xi(t)e^{-h^*t} \int_0^t s^{k-1}\di s = \frac{t}{k}\xi(t)
\]
and therefore, using \eqref{prop_h_2},
\[
\psi'' - \frac{h^*R + 1}{t} \psi' = \left[ h^* - \frac{k + h^* R}{t} \right]\psi' + \xi(t) \le -\frac{k}{t}\psi'(t) + \xi(t) \le 0
\]
on $(0,R)$, as claimed. Let $x \in B_R(x_0)\setminus \{x_0\}$ and $r = r(x)$. If $\psi''(r) \le \psi'(r)/r$, then using $h^* \ge h$ we deduce
\[
\begin{array}{lcl}
\PP_k^-(\nabla^2 w) - h|\nabla w| & = & \disp \psi''(r) + (k-1) \frac{\psi'(r)}{r} - h \psi'(r) \\[0.4cm]
& = & \disp \xi(r) + (h^*-h) \psi'(t) \ge \xi(r) \ge \frac{\xi(r)}{1+h^*R}.
\end{array}
 \]
On the other hand, if $\psi''(r) > \psi'(r)/r$, inequality \eqref{eq_claim} and $\psi' \ge 0$ give
\[
\begin{array}{lcl}
\PP_k^-(\nabla^2 w) - h|\nabla w| & = & \disp k \frac{\psi'(r)}{r} - h \psi'(r) \ge (k-1)\frac{\psi'(r)}{r} + \frac{\psi''(r)}{1+h^*R} - h \psi'(r) \\[0.4cm]
& = & \disp (k-1)\frac{\psi'(r)}{r} + \frac{1}{1+h^*R}\left[\xi(r) + h^* \psi'(r)-  \frac{k-1}{r}\psi'(r) \right] - h \psi'(r) \\[0.4cm]
& = & \disp \left[\frac{k-1}{r} \frac{h^*R}{1+h^*R} + \frac{h^*}{1+h^*R} - h\right] \psi'(r) + \frac{1}{1+h^*R}\xi(r) \\[0.4cm]
& \ge & \disp \frac{1}{R}\left[\frac{k h^*R}{1+h^*R} - hR\right] \psi'(r) + \frac{\xi(r)}{1+h^*R} \\[0.4cm]
& \ge & \disp \frac{\xi(r)}{1+h^*R}
\end{array}
\] 
pointwise on $B_R(x_0) \backslash \{x_0\}$, where the last inequality follows since our assumption $h^* \ge h/(k-hR)$ is equivalent to $\frac{k h^*R}{1+h^*R} - hR \ge 0$.\\
The $C^2$-regularity of $w$ and the validity of the pointwise inequality for $\FF_k^-[w]$ up to $x_0$ easily follow from the very definition of $\psi$.
\end{proof}

We next state our key Lemma, which refines \cite[Lem. 4.1]{beluma}.

\begin{lemma}\label{basicconstr_Pk}
Fix $x_0 \in \R^n$ and let $r(x)=|x-x_0|$. Fix $R>0$ and $k \in \{1,\ldots, n\}$, let $h, h^* \in \R^+_0$ satisfying
	\[
	hR < k, \qquad h^* \ge \max \left\{ h, \frac{h}{k-hR} \right\}.
	\]
Choose a non-negative, non-increasing function $S\in C(\R^+_0)$ satisfying
\begin{equation}\label{defS}
S =1 \quad \text{on } \, [0,1], \qquad \left\{ \begin{array}{ll}
\disp \int_0^{\infty} S(t) \di t = \hat{S} < \infty & \quad \text{if } k=1, \\[0.4cm]
\disp \int_0^{\infty} t S(t)\max\big\{1, |\log t|\big\} \di t = \hat{S} < \infty & \quad \text{if } k=2, \\[0.4cm]
\disp \int_0^{\infty} t S(t) \di t = \hat{S} < \infty & \quad \text{if } k>2. 
\end{array}\right.
\end{equation}
Then, there exists a positive constant $C_0=C_0(k, h^*R)$ such that the following holds: for each $a  \in (0, R/e]$, there is a $C^2$ function 
	\[
	u_{x_0} \colon B_R(x_0) \subset \R^n \ra \R 
	\]
such that
\begin{eqnarray}
(i) & & u_{x_0} \ge 0, \quad  u_{x_0}(x) =0 \ \text{ if and only if }
x =x_0; \label{propux1}\\[0.2cm]
(ii) & & \|u_{x_0}\|_{\infty} \le \left\{\begin{array}{ll} 
C_0 \hat S Ra & \quad \text{if } \, k=1, \\[0.2cm]
C_0 \hat S a^2 \log\left( \frac{R}{a}\right) & \quad \text{if } k = 2, \\[0.2cm]
C_0 \hat S a^2 & \quad \text{if } k>2;
\end{array}\right.
\label{propux2} \\[0.3cm]
(iii) & & \FF_k^-[u_{x_0}] \ge \frac{k S(r/a)}{1+h^*R} \quad \text{on } \, B_R(x_0)
\label{propux3}
\end{eqnarray}
where $\FF_k^-$ is as in \eqref{def_Fkpm}. In particular, $\FF_k^-[u_{x_0}] \ge \frac{k}{1+h^*R}$ on $B_a(x_0)$.
\end{lemma}

\begin{proof}
Define $\xi(t) = kS(t/a)$, and set
\begin{equation}\label{defing}
\psi(t) = \int_0^t \frac{e^{h^*s}}{s^{k-1}} \left[\int_0^s e^{-h^*\sigma}\sigma^{k-1} \xi(\sigma)
\di \sigma\right] \di s.
\end{equation}
Since $\psi$ solves \eqref{eq_h}, $\xi$ satisfies \eqref{inte_w} and $S$ is non-increasing and $1$ in a neighbourhood of zero, by Lemma \ref{lem_pk discend} the function $u_{x_0} = \psi(r)$ solves
$$
\FF_k^-[u_{x_0}] \ge \frac{\xi(r)}{1+h^*R} \qquad \text{on } \, B_R(x_0).
$$
To prove the $L^\infty$ bound, we change the order of integration and change variables to get, for $k=1$,
$$
\begin{array}{lcl}
\psi(t) & = & \disp \int_0^t e^{-h^*\sigma} \xi(\sigma)\left\{ \int_\sigma^t e^{h^*s}\di s \right\} \di  \sigma  \\[0.5cm]
& \le & \disp \int_0^t e^{h^*(t-\sigma)}(t-\sigma)\xi(\sigma) \di  \sigma \le \disp e^{h^*R} R \int_0^t \xi(\sigma) \di \sigma \\[0.5cm]
 & = & \disp e^{h^*R} Rk a \int_0^{t/a} S(\tau) \di \tau \le e^{h^*R} Rk \hat{S} a.
\end{array}
$$
For $k>2$,
$$
\begin{array}{lcl}
\psi(t) & = & \disp \int_0^t e^{-h^*\sigma}\sigma^{k-1} \xi(\sigma) \left\{ \int_\sigma^t \frac{e^{h^*s}\di s}{s^{k-1}} \right\}\di  \sigma  \\[0.5cm]
& \le & \disp \int_0^t e^{h^*(t-\sigma)}\sigma^{k-1} \xi(\sigma) \left\{ \int_\sigma^t \frac{\di s}{s^{k-1}} \right\}\di  \sigma  \\[0.5cm]
& \le & \disp e^{h^*R} \int_0^t \sigma^{k-1} \xi(\sigma)\frac{\sigma^{2-k}-t^{2-k}}{k-2} \di \sigma \\[0.5cm]
 & \le & \disp \frac{e^{h^*R}}{k-2} \int_0^t \sigma \xi(\sigma) \di  \sigma = \frac{a^2 k e^{h^*R}}{k-2} \int_0^{t/a} \tau S(\tau) \di \tau \le \frac{e^{h^*R}k \hat S}{k-2} a^2, 
\end{array}
$$
and for $k=2$,
\[
\begin{array}{lcl}
\psi(t) & = & \disp \int_0^t e^{-h^*\sigma} \sigma \xi(\sigma) \left\{ \int_\sigma^t \frac{e^{h^*s} \di s}{s} \right\}\di  \sigma  \le \int_0^t e^{h^*(t-\sigma)} \sigma \xi(\sigma)\log(t/\sigma) \di \sigma \\[0.5cm]
 & \le & \disp e^{h^*R}\int_0^R \sigma \xi(\sigma)\log(R/\sigma) \di \sigma \\[0.5cm]
& = & \disp e^{h^*R} \log R\int_0^R \sigma \xi(\sigma)\di \sigma - e^{h^*R}\int_0^R \sigma \xi(\sigma)\log \sigma \di \sigma \\[0.5cm]
& = & \disp k a^2 e^{h^*R} \log R \int_0^{R/a} \tau S(\tau) \di \tau - k a^2 e^{h^*R}\int_0^{R/a} \tau S(\tau)\log (a \tau) \di \tau \\[0.5cm]
& = & \disp k a^2 e^{h^*R} \log \left( \frac{R}{a}\right) \int_0^{R/a} \tau S(\tau) \di \tau - ka^2 e^{h^*R} \int_0^{R/a} \tau S(\tau)\log \tau \di \tau \\[0.5cm]
& \le & \disp k a^2 \log \left( \frac{R}{a}\right) e^{h^*R} \int_0^{R/a} \tau S(\tau) \di \tau + ka^2 e^{h^*R} \int_0^{R/a} \tau S(\tau)|\log \tau| \di \tau \\[0.5cm]
& \le & \disp k a^2 e^{h^*R} \left[ \int_0^{\infty} \tau S(\tau)\max\big\{1, |\log \tau|\big\} \di \tau \right] \left\{ \log \left( \frac{R}{a}\right) +1 \right\} \\[0.5cm]
& \le & \disp 2k e^{h^*R} \hat S a^2 \log \left( \frac{R}{a}\right), 
\end{array}
\]
where in the last line we used $a \le R/e$, so $\log(R/a) \ge 1$. This concludes the proof.
%\[
%\begin{array}
%& \le & a^2\int_0^{R/a} \tau S(\tau) \log(R/(a\tau)) \di \tau \\[0.5cm]
%& \le & \disp a^2 |\log R| \int_0^{R/a} \tau S(\tau) \di \tau + a^2|\log a| \int_0^{R/a} \tau S(\tau) \di %\tau + a^2 \int_0^{R/a} \tau S(\tau) |\log \tau| \di \tau \\[0.5cm]
% & \le & \hat S \big[|\log R| + |\log a| +1 \big] a^2 \le C a^2 |\log a|.
%\end{array}
%$$ 
\end{proof}

\begin{remark}
\emph{In \cite[Lem. 4.1]{beluma}, the radial function $u_{x_0}$ is constructed for $k \ge 2$, from a solution $\psi$ of 
	\[
	\psi''(t) + \frac{\theta}{t} \psi'(t) = \xi(t) \qquad \text{on } (0,R), 
	\]
where $\theta \doteq k-1 - hR$ is assumed to be positive (forcing the stronger requirement $hR < k-1$). In particular, the case $\theta \in (0,1)$ yields to an estimate on $u_{x_0}$ of the form $\|u_{x_0}\|_\infty \lesssim a^{\theta+1}$. As it will be apparent in the next theorem, the bound implies a more binding control on $E$ in terms of the Hausdorff measure $\haus^{\theta+1}$, leading to the stronger condition on $\lim \varphi \cap \Omega$ described in Remark \ref{rem_tech}.	
%	with the somehow puzzling feature that $hR$ appears in the measure. On the contrary, the use of $\psi$ as %defined in Lemma \ref{lem_pk discend} allows us to achieve a sharper and more natural estimate for $\|%u_{x_0}\|_\infty$ in terms of $a$, in particular, a bound whose power in $a$ is independent of $hR$.   	
}
\end{remark}

We are now ready to prove Theorem \ref{teo_maingoal_k2}, in the following strengthened form that will be used later to prove Theorem \ref{teo_main_geom}.

\begin{theorem}\label{teo_lemmafund}
Let $\Omega \subset \R^n$, $n \ge 3$ be an open subset with diameter $R$, fix $k \in \{1,\ldots, n-1\}$ and let $E \Subset \Omega$ be a compact subset satisfying $\haus_\Psi(E) < \infty$, where
	\[
	\Psi(t) = \left\{ \begin{array}{ll}
	Rt & \quad \text{if } \, k=1 \\[0.2cm]
	t^2 |\log(R/t)| & \quad \text{if } \, k=2 \\[0.2cm]
	t^2 & \quad \text{if } \, k \ge 3.
	\end{array}\right.
	\]
Fix $h \in \R^+_0$ satisfying $hR < k$. Then, there exists a constant $C_1= C_1(n,k,hR)$ with the following properties: for each $\eps$ and each $Q \in (\haus_\Psi(E), \infty)$, there exists a relatively compact, open set $U_\eps$ containing $E$ and there exists $w_\eps \in C^2(\overline\Omega)$ such that $w_\eps < 0$ on $\overline{\Omega}$ and 
	\[
	\FF_k^-[w_\eps] \ge \mathbb{1}_{U_\eps} \qquad \text{on } \, \Omega, \qquad \|w_\eps\|_\infty \le C_1 Q,
	\]
where $\FF_k^-$ is as in \eqref{def_Fkpm}. In particular, 
	\[
	\FF_k^-[w_\eps] + \frac{1}{C_1 Q} w_\eps \ge 0 \ \ \ \text{on } \, U_\eps, \qquad \FF_k^-[w_\eps] \ge 0 \ \ \text{on } \, \Omega,
	\]
and 
	\[
	\bar{\mu}(\FF_k^-,E) \ge \frac{1}{C_1\haus_\Psi(E)}.
	\]
\end{theorem}

\begin{proof}
Note first that $R > \mathrm{diam}(E)$. Choose $S(t) \in C^\infty_c([0,2))$, $S \equiv 1$ on $[0,1]$, and let  $\hat C_1$ denote the constant $C_0 \hat S$ in Lemma \ref{basicconstr_Pk}. Cover $E$ with a finite number of balls $\{B_i\}_{i=1}^t$, $B_i = B_{a_i}(x_i)$, $t=t(Q) \ge 2$ such that  
	\[
	x_i \in \Omega, \qquad 0 < a_i \le R/e, \qquad \sum_{j=1}^t \Psi(a_j) \le Q.
	\]
Define $h^* \doteq \max\{ h, h/(k - hR)\}$. To each $i$, let $u_i \doteq u_{x_i}$ given by Lemma \ref{basicconstr_Pk}, which is defined and $C^2$ on $\overline{B_R(x_i)} \supset \overline{\Omega}$. Define
	\[
	w_\eps = \frac{1 + h^*R}{k} \sum_i (u_i - 2\|u_i\|_\infty) \in C^2(\overline\Omega). 
	\]
Then, $w_\eps < 0$ on $\overline\Omega$ and there, by Lemma \ref{basicconstr_Pk} and the $1$-homogeneity and superadditivity of $\FF_k^-$, it satisfies  
	\[
	\FF_k^-[w_\eps] \ge 1 \qquad \text{on } \, \bigcup_i B_i \doteq U_\eps, \qquad \FF_k^-[w_\eps] \ge 0 \qquad \text{on } \, \Omega, 
	\]
and $\|w_\eps\|_\infty \le 2\frac{1+h^*R}{k} \hat C_1 Q \doteq C_1 Q$. Therefore, 
	\[
	\FF_k^-[w_\eps] + \frac{1}{C_1 Q} w_\eps \ge 0 \qquad \text{on } \, U_\eps, 
	\]	
showing that 
	\[
	\bar{\mu}(\FF_k^-,E) \ge \bar{\mu}(\FF_k^-,U_\eps) \ge \frac{1}{C_1 Q}.
	\]
The thesis follows by letting $Q \to \haus_\Psi(E)$. 	
\end{proof}

\begin{remark}
\emph{The constant $C_0$ in Lemma \ref{basicconstr_Pk}, hence $C_1$ in Theorem \ref{teo_lemmafund}, can be bounded from above in terms of $k,n$ and an upper bound for $h^*R$. Being
	\[
	h^*R \ge \max\left\{ hR, \frac{hR}{k-hR} \right\},
	\]
$C_0$ and $C_1$ can equivalently be bounded from above in terms of a lower bound for $k - hR$, as stated in Remark \ref{rem_constantC}. 
}
\end{remark}

%\begin{theorem}\label{teo_lemmafund}
%Let $E \Subset \Omega$ be a compact subset with $\haus_\Psi(E) < \infty$, and let $R= \mathrm{diam}(\Omega)$. Then, there exists a constant $C(R)>0$ and, for each sequence $\{c_j\}$ with 
%	\[
%	c_j \in \left( 0, \frac{1}{C \haus_\Psi(E)} \right),
%	\]	
%a sequence $\{U_j\}$ of relatively compact open sets and functions $w_j \in C^\infty(\Omega)$ such that $w_j < 0$ on $\Omega$ and 
%	\[
%	\FF_k^-(\nabla^2 w_j) + c_j w_j \ge 0 \ \ \ \text{on } \, U_j, \qquad \FF_k^-(\nabla^2 w_j) \ge 0 \ \ \text{on } \, \Omega.
%	\]
%	\end{theorem}

\section{From Theorem \ref{teo_maingoal_k2} to Theorem \ref{teo_main_geom}} \label{sec_3to5}

We premit a few observations. Given a Riemannian manifold $M$, its Laplace operator $\Delta_M$ is initially defined on $C^\infty_c(M)$, and then extended in a canonical way (Friedrichs extension) to a self-adjoint operator on a domain $\mathcal{D} \subset L^2(M)$. The spectrum $\sigma(-\Delta_M)$ is a closed subset of $\R^+_0$. Agreeing with the literature, we split $\sigma(-\Delta_M)$ into the discrete spectrum $\sigma_{\disc}(-\Delta_M)$ (eigenvalues with finite multiplicity, which are isolated in $\sigma(-\Delta)$) and the essential spectrum $\sigma_\ess(-\Delta_M)=\sigma(-\Delta_M)\backslash \sigma_{\disc}(-\Delta_M)$. For $\Omega \subset M$ open, let $\lambda(\Omega)$ be the bottom of the spectrum of the Friedrichs extension of $(-\Delta_M, C^\infty_c(\Omega))$, which coincides with the first eigenvalue if $\partial \Omega$ is Lipschitz. By Persson's formula \cite{persson},
$$
\inf \sigma_\ess(\Delta_M) = \sup_{K \subset M \text{ compact}} \lambda(M \backslash K).
$$
It is known by \cite{beniva} that $\lambda(M \backslash K)$ coincides with the principal eigenvalue of $M \backslash K$, defined as 
	\[
	\sup \Big\{ c \in \R \ : \ \exists v \in \USC(M \backslash K), \ v < 0  \ \ \text{on $M \backslash K$, } \ \Delta_M v + c v \ge 0 \ \text{on } \, M \backslash K \Big\}.
	\]
Therefore, to prove that $\sigma(-\Delta_M)$ is discrete, equivalently, that $\inf \sigma_\ess(-\Delta_M) = +\infty$, it is enough to produce, for each $\eps>0$, a compact set $K_\eps \subset M$ and functions $v_\eps < 0$ on $M \backslash K_\eps$ such that 
	\[
	\Delta_M v_\eps + C_\eps v_\eps \ge 0, \qquad \text{with } \, C_\eps \to +\infty \ \ \ \text{as } \, \eps \to 0.
	\]
To this aim, we first assume $(i)$, that is, that $\haus_\Psi(\lim \varphi) = 0$, and we define $h \doteq \|\mathbf{H}\|_\infty$. From Theorem \ref{teo_maingoal_k2}, we can take a sequence $\{U_\eps\}_{\eps>0}$ of relatively compact, open sets with $U_\eps \subset \Omega$ and  
$$
\lim \varphi \subset U_\eps, \qquad 2c_\eps = \bar \mu(\FF_k^-,U_\eps) \to +\infty \ \ \text{as } \, \eps \to 0.
$$
For each $\eps$, let $w_\eps \in \USC(\overline{U}_\eps)$, $w_\eps <0$ on $\overline{U}_\eps$ solve
$$
\FF^-_k[w_\eps] + c_\eps w_\eps \ge 0 \qquad \text{in viscosity sense on $U_\eps$.}
$$
Consider the functions $v_\eps = w_\eps \circ \varphi$. To explain the strategy, assume first that $w_\eps$ is $C^2$. Let $\{e_i\}$ be an orthonormal frame on $M$ in a neighbouhood of a point. Then, from the chain rule for the Hessian, the Laplacian $\Delta_M v_\eps$ of $v_\eps$ satisfies
\begin{equation}\label{eq_chain}
\begin{array}{lcl}
\disp \Delta_M v_\eps & = & \disp \sum_{i = 1}^k \nabla^2 w_\eps (\varphi_* e_i, \varphi_*e_i) + \langle \nabla w_\eps, H \rangle \\[0.4cm]
& \ge & \disp \sum_{i = 1}^k \nabla^2 w_\eps (\varphi_* e_i, \varphi_*e_i) - |H||\nabla w_\eps|.
\end{array}
\end{equation}
The term 
$$
\sum_{i = 1}^k \nabla^2 w_\eps (\varphi_* e_i, \varphi_*e_i)
$$
is the trace of $\nabla^2w_\eps$ restricted to the tangent plane $\varphi_* TM$ and thus, by the characterization in Remark \ref{rem_minmax}, it is at least $\PP^-_k(\nabla^2 w_\eps)$. Using the inequality satisfied by $w_\eps$ and $|H| \le h$ we therefore get 
	\[
	\Delta_M v_\eps \ge \PP^-_k(\nabla^2 w_\eps) - |H||\nabla w_\eps| \ge \FF_k^-[w_\eps] \ge - c_\eps w_\eps = -c_\eps v_\eps
	\]
on $\varphi^{-1}(U_\eps)$. Since $U_\eps$ contains $\lim \varphi$, $M \backslash \varphi^{-1}(U_\eps) = K_\eps$ is compact, so $\{v_\eps\}$ is the desired family of functions which guarantee the discreteness of $\sigma(-\Delta_M)$.
Next, we describe how to apply the above reasoning when $w_\eps$ has weak regularity. To this aim, we use Theorem 8.1 in \cite{HL_restriction}. We briefly explain their result in our setting, referring to \cite{HL_dir,HL_restriction} for notation and terminology. We consider the bundle of $2$-jets $J^2(M)$ and $J^2(\R^n)$, respectively over $M$ and $\R^n$. Jets $J \in J^2(\R^n)$ are denoted by $(y,r,p,A)$, where $y \in \R^n$, $r \in \R$, $p \in \R^n$ and $A \in \mathrm{Sym}^2(\R^n)$. We consider 
	\[
	\mathbf{F} = \Big\{ (r,p,A) \in \R \times \R^n \times \mathrm{Sym}^2(\R^n) : \ \PP_k^-(A) - h |p| + c_\eps r \ge 0\Big\}
	\]
and the subset
	\[
	F = \R^n \times  \mathbf{F} \subset J^2(\R^n)
	\]
which is, in Harvey and Lawson's terminology, a universal Riemannian subequation with model $\mathbf{F}$. In particular, $F$ is locally jet-equivalent modulo $M$ to $\mathbf{F}$. The differential inequality satisfied by $w_\eps$ is equivalent to say that $w_\eps$ is $F$-subharmonic on $U_\eps$ (namely, the $2$-jet of any $C^2$ function $\phi$ touching $w$ from above at a given point belongs to $F$). Consider the pull-back subset 
	\[
	\varphi^* F = \Big\{ \varphi^* J : J \in F \Big\} \subset J^2(M), 
	\] 
namely, if $J$ is the $2$-jet of the function $u$, then $\varphi^* J$ is the $2$-jet of the function $u \circ \varphi$. The computation in \eqref{eq_chain} guarantees that 
	\[
	\varphi^* F \subset G \doteq \Big\{ (x,s,q,B) \in J^2(M) \ : \ \tr(B) + c_\eps s \ge 0 \Big\}.
	\]
Note also that $G$ is a (universal, Riemannian) subequation on $M$. Then, the Restriction Theorem in \cite[Thm. 8.1]{HL_restriction} implies that $v_\eps$ is $\overline{\varphi^* F}$-subharmonic on $\varphi^{-1}(U_\eps)$, in particular, it is $G$-subharmonic. Equivalently, $v_\eps$ solves in the viscosity sense
	\[
	\Delta_M v_\eps + c_\eps v_\eps \ge 0 \qquad \text{on } \, \varphi^{-1}(U_\eps), 
	\]
as required. This concludes the proof in case $(i)$. To deal with case $(ii)$, we shall use the full strength of Theorem \ref{teo_lemmafund}, and also we shall produce a suitable barrier in a neighbourhood of $\partial \Omega$. First, because of \cite[Prop. 2]{galimame}, in the stated assumption \eqref{eq_barrierbordo} there exists a constant $\delta>0$ depending on 
	\[
	R, \ k, \ \inf_{\partial \Omega} \PP_k^-(\II_{\partial \Omega}) - \|\mathbf{H}\|_{\infty}, \ \inf_{\partial \Omega} \PP_{k-1}^-(\II_{\partial \Omega}),
	\]
and a function $w \in \lip(\overline{\Omega})$ such that $w=0$ on $\partial \Omega$, $w<0$ on $\Omega$ and 
	\[	
	\PP_k^-(\nabla^2 w) - h|\nabla w| \ge \delta \qquad \text{on } \, \Omega
	\]
in the barrier (hence, viscosity) sense. Hence, by the Restriction Theorem, $v \doteq w \circ \varphi$ satisfies in the viscosity sense
	\[
	\Delta_M v \ge \delta \qquad \text{on } \, M.
	\]
For $\eps>0$, define 
	\[
	V_\eps = \left\{ x \in \Omega : \mathrm{dist}(x, \partial \Omega) < \sqrt{\eps} \right\},
	\]
and let $E_\eps = \lim \varphi \cap (\Omega \backslash V_\eps)$. Note that $E_\eps$ is compact, and that $\haus_\Psi(E_\eps) = 0$, thus $\bar \mu( \FF_k^-,E_\eps) = +\infty$. By Theorem \ref{teo_lemmafund}, there exists a constant $C$ independent of $\eps$, an open set $U_\eps$ and a function $w_\eps \in C^2(\Omega)$, $w_\eps < 0$ on $\Omega$ satisfying 
	\begin{equation}\label{prop_wj_sharp}
	\FF_k^-[w_\eps] \ge \mathbb{1}_{U_\eps}, \qquad \|w_\eps\|_\infty \le C\eps.
	\end{equation}
%	\[	
%	+ c_j w_j \ge 0 \ \ \ \text{on } \, U_j, \qquad \FF_k^-(\nabla^2 w_j) \ge 0 \ \ \text{on } \, \Omega,
%	\]
Set $v_\eps = w_\eps \circ \varphi$, so that by restriction $\Delta_M v_\eps \ge \mathbb{1}_{\varphi^{-1}(U_\eps)}$, $\|v_\eps\|_\infty \le C\eps$. We study the function
	\[
	u_\eps \doteq v_\eps - C\eps + \sqrt{\eps} v \qquad \text{on } \, \varphi^{-1}(U_\eps \cup V_\eps).
	\]
Note that $K_\eps \doteq M \backslash \varphi^{-1}(U_\eps \cup V_\eps)$ is compact in $M$, and that 
	\[
	- 2 C\eps - \sqrt{\eps} |v| \le u_\eps \le - C\eps. 
	\]
On $\varphi^{-1}(U_\eps)$, we compute in the viscosity sense
	\[
	\begin{array}{lcl}
	\Delta_M u_\eps & = & \disp \Delta_M v_\eps + \sqrt{\eps} \Delta_M v \ge 1 + \sqrt{\eps} \delta \\[0,2cm]
	& \ge & \disp - \frac{u_\eps}{2 C\eps + \sqrt{\eps} \|v\|_\infty} + \sqrt{\eps}\delta \ge -\frac{C_1}{\sqrt{\eps}} u_\eps, 	
%	- \frac{1}{C\eps} (u_\eps + C\eps - \sqrt{\eps} v) + \sqrt{\eps} \delta \\[0.2cm]
%	& \ge & - \frac{1}{C\eps} (u_\eps + C\eps) + \sqrt{\eps} \delta \\[0.2cm]	
%	& \ge & - \frac{\sqrt{\eps}\delta}{C\eps} (u_\eps + C\eps) + \tau \delta \ge - \frac{\tau \delta}{C\eps} u_\eps
	\end{array}
	\]
for some constant $C_1= C_1(C, \|v\|_\infty,\delta)$. On the other hand, on $\varphi^{-1}(V_\eps)$, denoting with $L$ the Lipschitz constant of $w$ we deduce $|v|\le L\sqrt{\eps}$, hence
	\[
	\begin{array}{lcl}
	\Delta_M u_\eps & = & \disp \Delta_M v_\eps + \sqrt{\eps} \Delta_M v \ge \sqrt{\eps} \delta \\[0,2cm]
	& \ge & \disp -\frac{\sqrt{\eps} \delta}{2C\eps + \sqrt{\eps} v} u_\eps \ge -\frac{\delta}{(2C + L)\sqrt{\eps}} u_\eps \doteq -\frac{C_2}{\sqrt{\eps}} u_\eps.
	\end{array}
	\]
Summarizing, 
	\[
	\Delta_M u_\eps + \frac{\min\{C_1,C_2\}}{\sqrt{\eps}} u_\eps \ge 0 \qquad \text{on } \, \varphi^{-1}(U_\eps \cup V_\eps),
	\]
which implies $\inf \sigma_\ess(-\Delta) = +\infty$ by the arbitrariness of $\eps$.

\end{document}